\documentclass[a4paper,leqno,11pt]{amsart}

\usepackage{amssymb}

\theoremstyle{plain}
\newtheorem{thm}{Theorem}[section]
\newtheorem{lem}[thm]{Lemma}
\newtheorem{cor}[thm]{Corollary}
\newtheorem{remark}[thm]{Remark}

\theoremstyle{definition}
\newtheorem{defn}[thm]{Definition}

\newtheorem{notation}[thm]{Notation}
\newtheorem{exam}[thm]{Example}
\font\bigmath=cmsy10 scaled \magstep 2
\newcommand{\bigtimes}{\hbox{\bigmath \char'2}}
\begin{document}

\title{A Generalized Central Sets Theorem In Partial Semigroups}

\author{Arpita Ghosh}

\email{arpi.arpi16@gmail.com}

\email{arpitamath18@klyuniv.ac.in}

\address{Department of Mathematics, University of Kalyani,
Kalyani, Nadia-741235
West Bengal, India}

\begin{abstract}
The most powerful formulation of the Central Sets Theorem in an arbitrary semigroup was proved
in the work of De, Hindman, and Strauss. The sets which satisfy the
conclusion of the above Central Sets Theorem are called $C$-sets. The original
Central Sets Theorem was extended by J. McLeod for adequate commutative
partial semigroups. In this work, we will extend the Central Sets Theorem
obtained by taking all possible {\it adequate\/} sequences in a commutative adequate partial semigroup. 
We shall also discuss a sufficient conditions for being a set $C$-set in our context.
\end{abstract}

\maketitle

\section{introduction}

The notion of the central subset of $\mathbb{N}$ was originally introduced
by Furstenberg \cite{F81} in terms of a topological dynamical system. 
Before defining
central sets let us  start with original Central Sets Theorem
due to Furstenberg \cite{F81}.

\begin{thm}[The Original Central Sets Theorem]
 Let $l\in\mathbb{N}$ and for each $i\in\{1,2,\ldots,l\}$, and
let $\langle y_{i,n}\rangle_{n=1}^{\infty}$ be a sequence
in $\mathbb{Z}$. Let $C$ be a central subset of $\mathbb{N}$. Then
there exist sequences $\langle a_{n}\rangle_{n=1}^{\infty}$
in $\mathbb{N}$ and $\langle H_{n}\rangle_{n=1}^{\infty}$
in $\mathcal{P}_{f}(\mathbb{N})$ such that 
\begin{itemize}
\item[(1)] for all $n$, $\max H_{n} < \min H_{n+1}$ and
\item[(2)] for all $F\in\mathcal{P}_{f}(\mathbb{N})$ and all 
$i\in\{1,2,\ldots,l\},\sum_{n\in F}(a_{n}+\sum_{t\in H_{n}}y_{i,t})\in C$.
\end{itemize}
\end{thm}

\begin{proof}
\cite[Proposition 8.21]{F81}
\end{proof}

It was shown in \cite{BH90} that this definition was equivalent to 
a simpler algebraic characterization which we use
below. This algebraic characterization is in the setting of the Stone-\v{C}ech
compactification, $\beta S$, of a discrete semigroup $S$. 
We shall present a brief introduction to this structure
later in this section.

Observe that the conclusion of the Central Sets Theorem 
shows that central sets posses rich combinatorial structures. To present the notion central
sets we present a brief introduction of the algebraic structure of $\beta S$ for a
discrete semigroup $(S,+)$. We take the points of $\beta S$ to be
the ultrafilters on $S$, identifying the principal ultrafilters with
the points of $S$ and thus pretending that $S\subseteq\beta S$.
Given $A\subseteq S$ let us set, $\overline{A}=\{p\in\beta S\mid A\in p\}$.
 Then the set $\{\overline{A}\mid A\subseteq S\}$ is a basis for a topology
on $\beta S$. The operation $+$ on $S$ can be extended to the Stone-\v{C}ech
compactification $\beta S$ of $S$ so that $(\beta S,+)$ is a compact
right topological semigroup (meaning that for any $p\in\beta S$,
the function $\rho_{p}:\beta S\rightarrow\beta S$ defined by $\rho_{p}(q)=q+p$
is continuous) with $S$ contained in its topological center (meaning
that for any $x\in S$, the function $\lambda_{x}:\beta S\rightarrow\beta S$
defined by $\lambda_{x}(q)=x+q$ is continuous). Given $p,q\in\beta S$
and $A\subseteq S$, $A\in p+q$ if and only if $\{x\in S \mid -x+A\in q\}\in p$,
where $-x+A=\{y\in S \mid x+y\in A\}$.

A nonempty subset $I$ of a semigroup $(T,+)$ is called a left ideal
of ${T}$ if $T+I\subset I$, a right ideal if $I+T\subset I$,
and a two-sided ideal (or simply an ideal) if it is both a left and
a right ideal. A minimal left ideal is a left ideal that does not
contain any proper left ideal. Similarly, we can define a minimal right
ideal.

Any compact Hausdorff right topological semigroup $(T,+)$ has a smallest
two sided ideal

\[
\begin{array}{rl}
K(T) = & \bigcup\{L \mid L\text{ is a minimal left ideal of }T\}\\
  = &\bigcup\{R \mid R\text{ is a minimal right ideal of }T\}
\end{array}
\]

Given a minimal left ideal $L$ and a minimal right ideal $R$, $L\cap R$
is a group, and in particular, contains an idempotent. An idempotent
in $K(T)$ is called a minimal idempotent. If $p$ and $q$ are idempotents
in $T$, we write $p\leq q$ if and only if $p+q=q+p=p$. An idempotent
is minimal with respect to this relation if and only if it is a member
of the smallest ideal. See \cite{HS98}  for an elementary introduction
to the algebra of $\beta S$ and for any unfamiliar
details.

\begin{defn}
Let $S$ be a discrete semigroup and let $C$ be a subset of $S$.
Then $C$ is {\it central\/} if and only if there is some minimal idempotent $p \in \beta S$ with $p \in \overline{C}$.
\end{defn}

The original Central Sets Theorem was already strong enough to derive several combinatorial 
consequences such as Rado's theorem. Later this subject has been 
extensively studied by many authors.
In \cite{HS98} the Central Sets Theorem was extended to arbitrary semigroups. The version for 
commutative semigroups extended the original Central Sets Theorem by allowing the choice of countably many sequences
at a time.

\begin{thm}
Let $(S,\cdot)$ be a commutative semigroup. Let $C$ be a central subset
of $S$ and for each $l\in\mathbb{N}$, let $\left\langle y_{l,n}\right\rangle _{n=1}^{\infty}$
be a sequence in $S$. Then there exist sequences $\left\langle a_{n}\right\rangle _{n=1}^{\infty}$
in $S$ and $\left\langle H_{n}\right\rangle _{n=1}^{\infty}$ in
$\mathcal{P}_{f}(\mathbb{N})$ such that $ \max H_{n} < \min H_{n+1}$ for
each $n\in\mathbb{N}$ and such that for each $f\in\Phi$, the set
of all functions $f:\mathbb{N}\rightarrow\mathbb{N}$ for which $f(n)\leq n$
for all $n\in\mathbb{N}$, $FP({\left\langle a_{n}.\Pi_{t\in H_{n}}y_{f(n),t}\right\rangle }_{n=1}^{\infty})\break
\subseteq C$.
\end{thm}

Here in the above theorem, $FP$ stands for the finite products. Precisely, for a sequence 
$\langle x_n\rangle _{n= 1}^{\infty}$ in $S$, $FP(\langle x_n \rangle_{n= 1}^{\infty})=
\{\prod\limits_{n \in F}x_n \mid F \in \mathcal{P}_f(\mathbb{N})\}$.

Later, in \cite{DHS08}, the authors extended the Central Sets Theorem considering
all sequences at one time.

\begin{thm}
Let $(S,+)$ be a commutative semigroup and let $\mathcal{T}=S^{\mathbb{N}}$, 
the set of sequences in $S$. Let $C$ be a central subset of $S$.
There exist functions $\alpha:\mathcal{P}_{f}(\mathcal{T})\rightarrow S$
and $H:\mathcal{P}_{f}(\mathcal{T})\rightarrow\mathcal{P}_{f}(\mathbb{N})$
such that

\noindent (1) if $F,G\in\mathcal{P}_{f}(\mathcal{T})$ and $F\subsetneq G$,
then $\max H(F) < \min H(G)$ and

\noindent (2) whenever $m\in\mathbb{N},G_{1},G_{2},\ldots,G_{m}\in\mathcal{P}_{f}(\mathcal{T}),G_{1}\subsetneq G_{2}\subsetneq\ldots\subsetneq G_{m}$
and for each $i\in\{1,2,\ldots,m\}$ , $f_{i}\in G_{i}$, one has
$\sum_{i=1}^{m}(\alpha(G_{i})+\sum_{t\in H(G_{i})}f_i(t))\in C$.
\end{thm}

\begin{proof}
\noindent \cite[Theorem 2.2]{DHS08}.
\end{proof}

The authors in \cite{HS09}, introduced the notions of $C$-sets
which satisfying the conclusion of the above Central Sets Theorem.

\begin{defn}
Let $(S,+)$ be a commutative semigroup and let $A\subseteq S$ and
$\mathcal{T}=S^{\mathbb{N}}$, the set of sequences in $S$. The
set $A$ is a $C$-set if and only if there exist functions $\alpha:\mathcal{P}_{f}(\mathcal{T})\rightarrow S$
and $H:\mathcal{P}_{f}(\mathcal{T})\rightarrow\mathcal{P}_{f}(\mathbb{N})$
such that

\noindent (1) if $F,G\in\mathcal{P}_{f}(\mathcal{T})$ and $F\subsetneq G$,
then $ \max H(F) < \min H(G)$ and

\noindent (2) whenever $m\in\mathbb{N},G_{1},G_{2},\ldots,G_{m}\in\mathcal{P}_{f}(\mathcal{T}),G_{1}\subsetneq G_{2}\subsetneq\ldots\subsetneq G_{m}$
and for each $i\in\{1,2,\ldots,m\}$ , $f_{i}\in G_{i}$, one has
$\sum_{i=1}^{m}(\alpha(G_{i})+\sum_{t\in H(G_{i})}f_i(t))\in A$.
\end{defn}

The notion of partial semigroups is another important notion in the
study of Ramsey theory. In \cite{M05}, the author extended the Central
Sets Theorem for commutative adequate partial semigroup with a finite
number of adequate sequences. Before stating the theorem we
need to introduce some definitions.

\begin{defn}
(Partial semigroup) A partial semigroup is  defined as a pair $(G, \ast)$ where $\ast$ is an operation 
defined on a subset $X$ of $G \times G$ and satisfies the statement that for all 
$x$, $y$, $ z$ in $G$, $(x \ast y) \ast z = x \ast (y \ast z) $ in the sense that if either side 
is defined, so is the other and they are equal. A partial semigroup is commutative if ${x\ast y=y\ast x}$
for every ${(x,y)\in X}$.
\end{defn}

\begin{exam}
 Every  semigroup is a partial semigroup.
\end{exam}

\begin{exam}
 Let us consider ${G=\mathcal{P}_f({\mathbb{N}})}= \{F \mid \emptyset \neq F \subseteq \mathbb{N}$ and $F$ is finite$\}$ and let ${X=\{(\alpha,\beta)\in G \times G \mid \alpha\cap\beta=\emptyset\}}$
be the family of all pairs of disjoint sets, and let ${\ast:X\rightarrow G}$
be the union. It is easy to check that this is a commutative partial
semigroup. We shall denote this partial semigroup as $(\mathcal{P}_f({\mathbb{N}}),\uplus)$.

\end{exam}
\begin{defn}
Let $(S,\ast)$ be a partial semigroup.\\
(a) For $s\in S$, $\phi(s)=\{t\in S \mid s\ast t$ is defined$\}$.\\
(b) For $H\in\mathcal{P}_{f}(S)$, $\sigma(H)=\bigcap_{s \in H}\phi(s)$.\\
(c) $(S,\ast)$ is adequate if and only if $\sigma(H)\neq\emptyset$
for all $H\in\mathcal{P}_{f}(S)$.\\
(d) $\delta S=\bigcap_{x\in S}cl_{\beta S}(\phi(x))=\bigcap_{H\in\mathcal{P}_{f}(S)}cl_{\beta S}(\sigma(H))$.
\end{defn}

So, the partial semigroup $(\mathcal{P}_f({\mathbb{N}}),\uplus)$ is adequate. We are specifically interested in 
adequate partial semigroups as they lead to an interesting subsemigroup $\delta S$ of $\beta S$,
the Stone-\v{C}ech compactification of $S$ which  is itself
a compact right topological semigroup. Notice that adequacy of $S$
is exactly what is required to guarantee that $\delta S\neq\emptyset$.
If $S$ is, in fact, a semigroup, then $\delta S=\beta S$.

Now we recall some of the basic properties of the operation $\ast$
in $\delta S$.

\begin{defn}
Let $(S,\ast)$ be a partial semigroup. For $s\in S$ and $A\subseteq S$,
$s^{-1}A=\{t\in\phi(s) \mid s\ast t\in A\}$.
\end{defn}

\begin{lem}
Let $(S,\ast)$ be a partial semigroup, let $A\subseteq S$ and let
$a,b,c\in S$. Then $c\in b^{-1}(a^{-1}A)$ if and only if both $b\in\phi(a)$
and  $c\in(a\ast b)^{-1}A$.
In particular, if $b\in\phi(a)$, then $b^{-1}(a^{-1}A)=(a\ast b)^{-1}A$.
\end{lem}

\begin{proof}
\cite[Lemma 2.3]{HM01}
\end{proof}
\begin{defn}
Let $(S,\ast)$ be an adequate partial semigroup.\\
(a) For $a\in S$ and $q\in\overline{\phi(a)}$, $a\ast q=\{A\subseteq S \mid a^{-1}A\in q\}$.\\
(b) For $p\in\beta S$ and $q\in\delta S$, $p\ast q=\{A\subseteq S \mid \{a^{-1}A\in q\}\in p\}$.
\end{defn}

\begin{lem}
Let $(S,\ast)$ be an adequate partial semigroup.\\
(a) If $a\in S$ and $q\in\overline{\phi(a)}$, then $a\ast q\in\beta S$.\\
(b) If $p\in\beta S$ and $q\in\delta S$, then $p\ast q\in\beta S$.\\
(c) Let $p\in\beta S,q\in\delta S$, and $a\in S$. Then $\phi(a)\in p\ast q$
if and only if $\phi(a)\in p$.\\
(d) If $p,q\in\delta S$, then $p\ast q\in\delta S$.
\end{lem}

\begin{proof}
\cite[Lemma 2.7]{HM01}.
\end{proof}
\begin{lem}\label{1.12}
Let $(S,\ast)$ be an adequate partial semigroup and let $q\in\delta S$.
Then the function $\rho_{q}:\beta S\rightarrow\beta S$ defined by
$\rho_{q}(p)=p\ast q$ is continuous.
\end{lem}

\begin{proof}
\cite[Lemma 2.8]{HM01}.
\end{proof}
\begin{lem}
Let $p\in\beta S$ and let $q,r\in\delta S$. Then $p\ast(q\ast r)=(p\ast q)\ast r$.
\end{lem}

\begin{proof}
\cite[Lemma 2.9]{HM01}.
\end{proof}
\begin{defn}
Let $p=p \ast p \in \delta S$ and let $A \in p$. Then $A^{\ast} = \{x \in A \mid x^{-1}A \in p\}$.
\end{defn}

Given an idempotent $p \in \delta S$ and $A \in p$, it is immediate that $A^{\ast} \in p$.
\begin{lem}\label{1.15}
Let $p= p \ast p \in \delta S$, let $A \in p$, and let $x \in A^{\ast}$. Then $x^{-1}A^{\ast} \in p$.
\end{lem}

\begin{proof}
\cite[Lemma 2.12]{HM01}.
\end{proof}

As a consequence of the above results, we have that if $(S, \ast)$ is an adequate 
partial semigroup, then $(\delta S, \ast)$ is a compact right topological semigroup. 
Being a compact right topological semigroup, $\delta S$ contains
idempotents, left ideals, a smallest two-sided ideal, and minimal idempotents. 
Thus $\delta S$ provides a suitable environment for considering the notion of central sets.

\begin{defn} Let $(S,*)$ be an adequate partial semigroup. A set 
$C\subseteq S$ is {\it central\/} if and only if there is an idempotent
$p\in\overline C\cap K(\delta S)$.\end{defn}

\begin{defn}
Let $(S,\ast)$ be an adequate partial semigroup and let\break
 $\left\langle y_{n}\right\rangle _{n=1}^{\infty}$
be a sequence in $S$. Then $\left\langle y_{n}\right\rangle _{n=1}^{\infty}$
is adequate if and only if $\prod_{n\in F}y_{n}$ is defined for each
$F\in\mathcal{P}_{f}(\mathbb{N})$ and for every $K\in\mathcal{P}_{f}(S)$,
there exists $m\in\mathbb{N}$ such that $FP(\left\langle y_{n}\right\rangle _{n=m}^{\infty})\subseteq\sigma(K)$.
\end{defn}

The following is the Central Sets Theorem for commutative adequate partial semigroup with a 
finite number of sequences proved in \cite{M05}. 

\begin{thm}\label{finitepcst}
Let $(S,\ast)$ be a commutative adequate partial semigroup  and let $C$ be a central subset of 
$S$. Let $k \in \mathbb{N}$ and for each $l\in\{1,2,\ldots,k\}$, let $\left\langle y_{l,n}\right\rangle _{n=1}^{\infty}$ 
be an adequate 
sequence in $S$. There exists a sequence $\left\langle a_{n}\right\rangle _{n=1}^{\infty}$ in $S$ and a 
sequence $\left\langle H_{n}\right\rangle _{n=1}^{\infty}$ in $\mathcal{P}_{f}(\mathbb{N})$ such that $\max H_{n} < \min H_{n+1}$ for each $n\in\mathbb{N}$ and such that for each $f:\mathbb{N}\rightarrow\{1,2,\ldots,k\}$, $$FP(\langle a_{n}\ast\prod_{t\in H_{n}}y_{f(n),t}\rangle _{n=1}^{\infty})\subseteq C.$$
\end{thm}

\begin{proof} \cite[Theorem 3.4]{M05} \end{proof}

Section 2 of this paper deals with a version of Central Sets Theorem for commutative 
adequate partial semigroup which deals with all adequate sequences in $S$ 
at once and we show that, as long as $S$ has infinitely many
adequate sequences, this result generalizes Theorem \ref{finitepcst}.
In Section 3, we define a $C$-set in an adequate partial semigroup to be
a set satisfying the conclusion of this new Central Sets Theorem   and establish a sufficient condition for being a $C$-set.

\begin{notation} Through out this document $\overline{A}$ stands for the 
closure of $A$ with respect to $\beta S$. Also, we use the 
notation $cl_{Y}(A)$ to denote the closure of $A$ with 
respect to $Y$, whenever $Y$ is other than $\beta S$.
\end{notation}

\section{New Central Sets Theorem For Commutative Adequate Partial Semigroups}

In this section, we establish a version of the Central Sets Theorem for commutative
adequate partial semigroups that applies to all adequate sequences at once. We start with some relevant definitions.

\begin{defn}
Let $W_1, W_2 \in \mathcal{P}_f(S)$, then define $W_1 \ast W_2=
\{w_1 \ast w_2 \mid w_1 \in W_1,  w_2 \in W_2 \; \text{and} \; w_1 \ast w_2 \;  \text{is defined}\}$.
\end{defn}

Now we recall the following lemma from \cite[Theorem 3.4]{M05}. For the sake of completeness we include the proof here.

\begin{lem}\label{conv}
Let $(S, \ast)$ be a commutative adequate partial semigroup, let $m$, $r\in\mathbb{N}$, and 
for each $i \in \{1,2,\ldots,m\}$,  let $f_i$ be an adequate sequence in $S$. Let $p$ be a 
minimal idempotent in $\delta S$, let $B\in p$, and let $W\in\mathcal{P}_{f}(S)$. There
exist $a\in\sigma(W)$ and $L\in\mathcal{P}_{f}(\mathbb{N})$ such that 
$\min L>r$, and for each $i\in\{1,2,\ldots,m\}$, $\prod_{t\in L} f_i(t)\in\phi(a)$ and 
$a\ast\prod_{t\in L}f_i(t)\in B$.
\end{lem}

\begin{proof}
Let $\mathcal{D}=\mathcal{P}_f(S)\times \{r+1,r+2, \ldots\}$ be a directed set with ordering 
defined by $(W_2,n_2) \geq (W_1,n_1)$ if $W_1 \subseteq W_2$ and $n_1 \leq n_2$. For each $(W,n) \in \mathcal{D}$ 
we define $I_{(W,n)} \subseteq \bigtimes_{i=1}^m\sigma(W)$ to be the set of 
elements of the form $\big(a \ast \prod_{t \in L}f_1(t), a \ast \prod_{t \in L}f_2(t), \ldots, a \ast 
\prod_{t \in L}f_{m}(t)\big)$, such that $a \in \sigma(W)$ and $L \in \mathcal{P}_f(\mathbb{N})$ 
with $\min L > n$ and $ \prod_{t \in L}f_i(t) \in \sigma(W \ast a)$ for every 
$i \in \{1,2,\ldots,m\}$. Define also the set $E_{(W,n)}=I_{(W,n)} 
\cup \{(a, a, \ldots, a) : a \in \sigma(W)\}$.

Let $Y= \bigtimes_{i=1}^m \beta S$. Let $$I= \bigcap_{(W,n)\in \emph{D}}cl_Y I_{(W,n)}
\text{ and let }E= \bigcap_{(W,n)\in \emph{D}}cl_Y E_{(W,n)}\,.$$
To see that $I\neq\emptyset$, note that if $(W_1,n_1)$ and $(W_2,n_2)$ are in ${\mathcal D}$, 
then $I_{(W_1 \cup W_2, \max \{n_1, n_2\})} 
\subseteq I_{(W_1,n_1)} \cap I_{(W_2,n_2)}$. It thus suffices to let
$(W,n)\in{\mathcal D}$ and show that $I_{(W,n)}\neq\emptyset$.
Pick $a \in \sigma (W)$ and for each $i \in \{1,2,\ldots,m \}$, pick a natural 
number $k_i$ such that $FP(\langle f_i(t) \rangle_{t=k_i}^{\infty}) \subseteq 
\sigma (W \ast a)$. Let $k=\max\big\{k_i\mid i\in\{1,2,\ldots,m\}\big\}$. Then
$\big(a*f_1(k),a*f_2(k),\ldots,a*f_m(k)\big)\in I_{(W,n)}$.
Since $E_{(W,n)} \subseteq \bigtimes_{i=1}^m \sigma(W)$, that 
implies $E \subseteq \bigtimes_{i=1}^m \delta S$. Therefore, we have 
$\emptyset \neq I \subseteq E \subseteq \bigtimes_{i=1}^m \delta S$.

Let $\vec{q}$, $\vec{r}\in E$. We show that $\vec{q} \ast \vec{r} \in E$ 
and if either $\vec{q} \in I$ or $\vec{r} \in I$, then $\vec{q} \ast \vec{r} 
\in I$. Let $(W, n) \in \mathcal{D}$. Let $U$ be a neighbourhood of $\vec{q} 
\ast \vec{r}$ in $Y$. So for each $i \in \{1,2, \ldots, m\}$, pick $A_i \in 
q_i \ast r_i$ such that $\bigtimes_{i=1}^m \overline{A_i} \subseteq U$. By Lemma \ref{1.12} pick 
for each $i$, $B_i \in q_i$ such that $\rho_{r_i}[\,\overline{B_i}\,] \subseteq \overline{A_i}$. 
Let $V= \bigtimes_{i=1}^m \overline{B_i}$. Now pick $\vec{x} \in E_{(W,n)} \cap V$ with 
$\vec{x} \in I_{(W,n)}$ if $\vec{q} \in I$.
 
  If $\vec{x} \in I_{(W,n)}$,  take $\vec{x}=(a \ast \prod_{t \in L}f_1(t),   a \ast \prod_{t \in L}f_2(t),  
\ldots,  a \ast \prod_{t \in L}f_{m}(t))$, where $a \in \sigma(W)$ and $L \in \mathcal{P}_f(\mathbb{N})$ such that 
$\min L > n$ and $\prod_{t \in L}f_i(t) \in \sigma(W \ast a)$ for every $i \in \{1,2,\ldots,m\}$. 
Let $W^{\prime}=W \ast F$, where $$F=\{a,   a \ast \prod_{t \in L}f_1(t),   
a \ast \prod_{t \in L}f_2(t),  \ldots,  a \ast \prod_{t \in L}f_{m}(t)\}\,,$$ and let $n^{\prime}= \max L$. 
Then $(W^{\prime} ,  n^{\prime}) \in \mathcal{D}$. We claim that 
$\vec{x} \ast E_{(W^{\prime}, n^{\prime})} \subseteq I_{(W,n)}$. Let 
$\vec{y} \in \vec{x} \ast E_{(W^{\prime}, n^{\prime})}$. So, $\vec{y}= \vec{x} \ast \vec{z}$ for 
some $\vec{z} \in E_{(W^{\prime}, n^{\prime})}=I_{(W^{\prime}, n^{\prime})} \cup 
\{(a^{\prime}, a^{\prime}, \ldots , a^{\prime}) \mid a^{\prime} \in \sigma (W^{\prime})\}$. 
First we consider  $\vec{z}=(a^{\prime}, a^{\prime}, \ldots , a^{\prime})$ where   $a^{\prime} \in \sigma(W^{\prime})$. 
Thus we obtain $\vec{y}= (a \ast \prod_{t \in L}f_1(t) \ast a^{\prime},  
a \ast \prod_{t \in L}f_2(t) \ast a^{\prime},    \ldots,\break     
a \ast \prod_{t \in L}f_m(t) \ast a^{\prime})$. Since $a^{\prime} \in \sigma(W^{\prime})= \sigma(W \ast F)$, 
therefore, for all $w \in W$ and for each $i \in \{1,2, \ldots,m\}$, 
$w \ast a \ast \prod_{t \in L}f_i(t) \ast a^{\prime}$ is defined. This gives 
$a \ast a^{\prime} \ast \prod_{t \in L}f_i(t)=a \ast \prod_{t \in L}f_i(t) \ast a^{\prime} \in \sigma(W)$ 
and $a \ast a^{\prime} \in \sigma (W)$. Thus we obtain $\vec{y} \in I_{(W,n)}$. 
Next, we take $\vec{z} \in I_{(W^{\prime}, n ^{\prime})}$. Thus we assume $\vec{z} = 
(a^{\prime} \ast \prod_{t \in L^{\prime}}f_1(t),   a^{\prime} \ast \prod_{t \in L^{\prime}}f_2(t),   
\ldots,  a^{\prime} \ast \prod_{t \in L^{\prime}}f_m(t))$, where $a^{\prime} \in 
\sigma (W^{\prime})$ and $L^{\prime} \in \mathcal{P}_f(\mathbb{N})$ satisfying $\min L^{\prime} 
> n^{\prime}$ and $\prod_{t \in L^{\prime}}f_i(t) \in \sigma(W^{\prime} \ast a^{\prime})$ for 
each $i \in \{1,2, \ldots, m\}$. Then $\vec{y}= (a \ast \prod_{t \in L}f_1(t) \ast a^{\prime} \ast 
\prod_{t \in L^{\prime}}f_1(t),   a \ast \prod_{t \in L}f_2(t) \ast a^{\prime} \ast 
\prod_{t \in L^{\prime}}f_2(t),   \ldots,   a \ast \prod_{t \in L}f_m(t) \ast a^{\prime} 
\ast \prod_{t \in L^{\prime}}f_m(t))$, i.e., 
$$\vec{y}= (a \ast a^{\prime} \ast \prod_{t \in L \cup L^{\prime}}f_1(t),   
a \ast a^{\prime} \ast \prod_{t \in L \cup L^{\prime}}f_2(t),   \ldots,   
a  \ast a^{\prime} \ast \prod_{t \in L \cup L^{\prime}}f_m(t))\,.$$ For all $w \in W$, and for 
each $i \in \{1,2, \ldots,m\}$, $w \ast a \ast a^{\prime} \ast \prod_{t \in L \cup L^{\prime}}f_i(t)$ 
is defined and so $a \ast a^{\prime} \ast \prod_{t \in L \cup L^{\prime}}f_i(t) \in \sigma (W)$ 
and $a \ast a^{\prime} \in \sigma (W)$. Also note that $\min (L \cup L^{\prime}) > n$. Thus 
$\vec{y} \in I_{(W,n)}$. Therefore, we have $\vec{x} \ast E_{(W^{\prime}, n^{\prime})} \subseteq I_{(W,n)}$. 
 
If $\vec{x} \in E_{(W,n)}\setminus I_{(W,n)}$, we take $\vec{x}=(a, a,  \ldots,  a)$ where $a \in \sigma(W)$. Let 
$W^{\prime}=W \ast a $ and $n^{\prime}=n$. Then $(W^{\prime} ,  n^{\prime}) \in \mathcal{D}$. Our next claim is 
that $\vec{x} \ast E_{(W^{\prime} ,  n^{\prime})} \subseteq E_{(W,n)}$ and $\vec{x} \ast I_{(W^{\prime} ,  n^{\prime})} 
\subseteq I_{(W,n)}$. Let $\vec{y} \in \vec{x} \ast E_{(W^{\prime}, n^{\prime})}$, then, 
$\vec{y}= \vec{x} \ast \vec{z}$,  for some $\vec{z} \in E_{(W^{\prime}, n^{\prime})}= 
I_{(W^{\prime}, n^{\prime})} \cup \{(a^{\prime}, a^{\prime},  \ldots,   a^{\prime} ) \mid a^{\prime} 
\in \sigma (W^{\prime})\}$. If $\vec{z} = (a^{\prime},  a^{\prime},  \ldots,   
a^{\prime})$ for some $a^{\prime} \in \sigma (W^{\prime})$. Readily, 
$\vec{y}=(a \ast a^{\prime},   a \ast a^{\prime},   \ldots,   a \ast a^{\prime})$ and 
$a \ast a^{\prime} \in \sigma (W)$ by the construction of $W^{\prime}$. Now consider 
$\vec{z} \in I_{(W^{\prime}, n^{\prime})}$. Then $\vec{z}= (b \ast \prod_{t \in L^{\prime}} f_1(t),   
b \ast \prod_{t \in L^{\prime}} f_2(t),   \ldots,   b \ast \prod_{t \in L^{\prime}} f_m(t)) $, 
where $b \in \sigma (W^{\prime})$ and $L^{\prime} \in \mathcal{P}_f(\mathbb{N})$ satisfying 
$\min L^{\prime} > n^{\prime}$ and $b \ast \prod_{t \in L^{\prime}} f_i(t) \in \sigma (W^{\prime})$ 
for all $i \in \{1,2, \ldots, m\}$. Now $\vec{y}= (a \ast b \ast \prod_{t \in L^{\prime}}f_1(t),  
a \ast b \ast \prod_{t \in L^{\prime}}f_2(t),   \ldots,   a \ast b \ast \prod_{t \in L^{\prime}}f_m(t ))$. 
Since for each $i \in \{1,2,\ldots, m\}$, $b \ast \prod_{t \in L^{\prime}}f_i(t) \in \sigma (W \ast a)$  
then $a \ast b \ast \prod_{t \in L^{\prime}}f_i(t) \in \sigma (W) $ and $a \ast b \in \sigma (W)$. 
Therefore, $\vec{x} \ast E_{(W^{\prime} ,  n^{\prime})} \subseteq E_{(W,n)}$ and $\vec{x} \ast I_{(W^{\prime} ,  
n^{\prime})} \subseteq I_{(W,n)}$.
 
Since $\vec{x} \ast \vec{r} \in \bigtimes_{i=1}^m \overline{A_i}$, we have for each $i \in \{1,2, \ldots ,m\}$ 
that $x_i \ast r_i \in \overline{A_i}$ so $x_i^{-1} A_i \in r_i$. 
Then $O= \bigtimes_{i=1}^m \overline{x_i^{-1}A_i}$ is a neighbourhood 
of $\vec{r}$ so pick $\vec{y} \in O \cap E_{(W^{\prime}, n^{\prime})}$ with 
$\vec{y} \in I_{(W^{\prime}, n^{\prime})}$ if $\vec{r} \in I$.
Then $\vec{x} \ast \vec{y} \in U \cap E_{(W,n)}$ and if either $\vec{q} \in I$ or $\vec{r} \in I$, 
then $\vec{x} \ast \vec{y} \in U \cap I_{(W,n)}$. This gives $E$ is a subsemigroup of $\bigtimes_{i=1}^m \delta S$ 
and $I$ is an ideal of $E$.

Now $p \in K(\delta S)$. Our next claim is that $\overline{p}=(p, p, \ldots, p) \in E$. Let $U$ be a 
neighbourhood of $\overline{p}$ in $Y$ and let $(W,n)\in{\mathcal D}$. Pick $C \in p$ such that $\bigtimes_{i=1}^m \overline{C} \subseteq U$. 
Now $p \in \delta S \subseteq cl_{\beta S} \sigma(W)$,  therefore, $\sigma(W) \in p$. So, we can pick $a \in C \cap \sigma(W)$,  
and therefore, $(a, a, \ldots, a) \in U  \cap E_{(W,n)}$. 
By \cite[Theorem 2.23]{HS98} we have that $ \bigtimes_{i=1}^m K(\delta S)=K(\bigtimes_{i=1}^m \delta S)$, 
so $\overline{p} \in E \cap K(\bigtimes_{i=1}^m \delta S). $ Now by \cite[Theorem 1.65]{HS98}, we have that 
$\overline{p} \in K(E)$ and, since $I$ is an ideal of $E$, we have $\overline{p} \in I$. Since 
$\bigtimes_{i=1}^m \overline{B}$ is a neighbourhood of $\overline{p}$, then we have 
$\bigtimes_{i=1}^m \overline{B} \cap I_{(W,n)} \neq \emptyset$.
\end{proof}

Now we are in a situation to prove a stronger version of the Central Sets Theorem for 
commutative adequate partial semigroups.

\begin{defn} Let $(S,\ast)$ be a commutative adequate partial semigroup.
Then $\mathcal{T}$ is the set of all adequate sequences in $S$.
\end{defn}

\begin{thm}\label{adeprod}
Let $(S,\ast)$ be a commutative adequate partial semigroup and let $C$
be a central subset of $S$. There exist functions $\alpha:\mathcal{P}_{f}(\mathcal{T})\rightarrow S$
and $H:\mathcal{P}_{f}(\mathcal{T})\rightarrow\mathcal{P}_{f}(\mathbb{N})$
such that
\begin{itemize}
\item[(1)] $F,G\in\mathcal{P}_{f}(\mathcal{T})$ and $F\subsetneq G$,
then $\max H(F) < \min H(G)$ and
\item[(2)] whenever $m\in\mathbb{N}$, $G_{1},G_{2},\ldots,G_{m}\in\mathcal{P}_{f}(\mathcal{T})$, 
$G_{1}\subsetneq G_{2}\subsetneq\ldots\subsetneq G_{m}$, and for each 
$i\in\{1,2,\ldots,m\}$, $f_i \in G_{i}$, one 
has\hfill\break 
$\prod_{i=1}^{m}(\alpha(G_{i})\ast\prod_{t\in H(G_{i})}f_i(t))\in C$.
\end{itemize}
\end{thm}

\begin{proof}
Let $p$ be a minimal idempotent in $\delta S$, let $C\in p$, and let 
$C^{\star}=\{x\in C:x^{-1}C\in p\}$. Then $C^{\star}\in p$ and by Lemma \ref{1.15}, if 
$x\in C^{\star}$, then $x^{-1}C^{\star}\in p$.

We define $\alpha(F)\in S$ and $H(F)\in\mathcal{P}_{f}(\mathbb{N})$
for $F\in\mathcal{P}_{f}(\mathcal{T})$ by using induction
on $|F|$, the cardinality of $F$, satisfying the following
inductive hypotheses:
\begin{itemize}
\item[(1)] if $\emptyset\neq G\subsetneq F$, then $\max H(G) < \min H(F)$
and 
\item[(2)] if $n\in\mathbb{N}$, $\emptyset\neq G_{1}\subsetneq G_{2}\subsetneq\ldots\subsetneq G_{n}=F$,
and $\langle f_{i}\rangle_{i=1}^{n}\in\bigtimes_{i=1}^{n}G_{i}$, then
$\prod_{i=1}^{n}(\alpha(G_{i})\ast\prod_{t\in H(G_{i})}f_i(t))\in C^{\star}$.
\end{itemize}

This will suffice to prove the theorem.

 First we assume that $|F|=1$ so that $F=\{f\}$ for some $f \in \mathcal{T}$. By 
Lemma \ref{conv} pick $a\in S$ and $L\in\mathcal{P}_{f}(\mathbb{N})$
such that $\prod_{t\in L}f(t)\in\phi(a)$ and $a\ast\prod_{t\in L}f(t)\in C^{\star}$.
Put $\alpha(F)=a$ and $H(F)=L$. Hypothesis (1) holds vacuously. To verify hypothesis (2), 
let $n \in \mathbb{N}$, $\emptyset \neq G_1 \subsetneq G_2 \subsetneq \ldots \subsetneq G_n=F$, and 
$\langle f_i \rangle_{i=1}^n \in \bigtimes_{i=1}^n G_i$. Then necessarily $n=1$ 
and $f_n=f$, so $ \prod_{i=1}^{n}(\alpha(G_{i})\ast\prod_{t\in H(G_{i})}f_i(t)) 
=a \ast \prod_{t \in L} f(t) \in C^\star$.

Now assume that $|F|>1$ and $\alpha(G)$ and $H(G)$ have been defined for every nonempty proper subset 
$G$ of $F$. Let $K=\bigcup\{H(G):\emptyset\neq G\subsetneq F\}$
and let $r=\max K$. Define
\[
\begin{array}{rl}
M= & \big\{\prod_{i=1}^{n}(\alpha(G_{i})\ast\prod_{t\in H(G_{i})}f_{i}(t)) \mid n\in\mathbb{N},\\
& \emptyset\neq G_{1}\subsetneq G_{2}\subsetneq\ldots\subsetneq G_{n}\subsetneq F,\mbox{and}\langle f_{i}\rangle_{i=1}^{n}\in\bigtimes_{i=1}^{n}G_{i}\big\}.
\end{array}
\]

Then by hypothesis (2), $M$ is a finite subset of $C^\star$. Let $A=C^{\star}\cap\bigcap_{x\in M}x^{-1}C^{\star}$.
Then $A\in p$. Using Lemma \ref{conv}, pick $a\in\sigma(M)$ and $L\in\mathcal{P}_{f}(\mathbb{N})$
such that $\min L>r$, and for each $f\in F$, $\prod_{t\in L}f(t)\in\phi(a)$ and $a\ast\prod_{t\in L}f(t)\in A$. 
Let $\alpha(F)=a$ and $H(F)=L$.

Since $\min L > r$, hypothesis (1) holds. To verify hypothesis (2), let $n\in\mathbb{N}$,
let $\emptyset\neq G_{1}\subsetneq G_{2}\subsetneq\ldots\subsetneq G_{n}=F$,
and let $\langle f_{i}\rangle_{i=1}^{n}\in\bigtimes_{i=1}^{n}G_{i}$. If $n=1$, then
$\prod_{i=1}^{n}(\alpha(G_{i})\ast\prod_{t\in H(G_{i})}f_{i}(t))
=a\ast\prod_{t\in L}f(t)\in A \subseteq C^\star$.

 Now assume that $n>1$ and let $y = \prod_{i=1}^{n-1}(\alpha(G_{i})\ast\prod_{t\in H(G_{i})}f_{i}(t))$.
Then $y\in M$, so  $a \in \sigma(M) \subseteq \phi(y)$ and $a \ast \prod_{t \in L} f_n(t) \in A \subseteq y^{-1}C^\star$. Therefore, $\prod_{i=1}^{n}(\alpha(G_{i})\ast\prod_{t\in H(G_{i})}f_{i}(t))= y \ast a \ast \prod_{t \in L}f_n(t) \in C^\star$.
\end{proof}

We show now that Theorem \ref{adeprod} does generalize \cite[Theorem 3.4]{M05}, assuming only
that there are infinitely many adequate sequences in $S$.  Unless $|S|=1$, there are certainly
infinitely many sequences. And it is easy to show that if $S$ is countably infinite, then for
each $a\in S$, there is an adequate sequence $f$ with $f(1)=a$.

\begin{cor}
Let $(S, \ast)$ be a commutative adequate partial semigroup and assume that ${\mathcal T}$ is infinite. 
Let $k \in \mathbb{N}$, for 
each $l \in \{1,2, \ldots, k\}$ let $f_l \in \mathcal{T}$, and let $C$ be central  
in $S$. There exist a sequence $\langle a_n \rangle_{n=1}^{\infty}$ in $S$ and a sequence 
$\langle H_n \rangle_{n=1}^{\infty}$ in $\mathcal{P}_f(\mathbb{N})$ such that  
$\max H_n < \min H_{n+1}$ for all $n$ and for each $g : \mathbb{N} \rightarrow 
\{1,2, \ldots, k\}$,  $FP(\langle a_n \ast \prod_{t \in H_n}f_{g(n)}(t) 
\rangle_{n=1}^{\infty}) \subseteq C$.
\end{cor}

\begin{proof} Pick functions $\alpha$ and $H$ as guaranteed by Theorem \ref{adeprod} for $C$.
Now choose for each $n\in{\mathbb N}$, $\gamma_n \in \mathcal{T} \setminus \{f_1, \ldots, f_k \}$ such 
that $\gamma_n \neq \gamma_m$ if $n \neq m$. 
For $n \in \mathbb{N}$ define $G_n = \{f_1,f_2 \ldots, f_k \} \cup \{\gamma_1, \gamma_2,\ldots, \gamma_n\}$.
Also, define $a_n = \alpha(G_n)$ and $H_n = H(G_n)$ for $n \in \mathbb{N}$.
Since $G_n\subsetneq G_{n+1}$ for each $n$, we have $\max H_n<\min H_{n+1}$ for each $n$.

Now let $g : \mathbb{N} \to \{1, 2, \ldots, k \}$ and let $F\in{\mathcal P}_f({\mathbb N})$.
Enumerate $F$ in order as $\langle i(j)\rangle_{j=1}^m$.
Then $\emptyset\neq G_{i(1)}\subsetneq G_{i(2)}\subsetneq\ldots\subsetneq G_{i(m)}$, and 
for each $j\in\{1,2,\ldots,m\}$, $f_{g(i(j))}\in G_{i(j)}$ so
$$\prod_{n\in F}\big(a_n \ast \prod_{t \in H_n}f_{g(n)}(t)\big)=
\prod_{j=1}^m\big(\alpha(G_{i(j)})*\prod_{t\in H(G_{i(j)})}f_{g(i(j))}(t)\big)\in C\,.$$
\end{proof}

\begin{remark}
The author has recently learned that a result nearly identical to Theorem {\rm \ref{adeprod}} has 
been independently proved by K. Pleasant in \cite{KP17}.

\end{remark}
\section{Sets Satisfying The New Central Sets Theorem}

We start this section by defining a new notion of sets which satisfy the conclusion of Theorem \ref{adeprod}.
\begin{defn}\label{cset}
Let $(S,\ast)$ be a commutative adequate partial semigroup. Then a set $A$ is said to be a $C$-set if and only if there exist functions
$\alpha:\mathcal{P}_{f}(\mathcal{T})\rightarrow S$ and $H:\mathcal{P}_{f}(\mathcal{T})\rightarrow\mathcal{P}_{f}(\mathbb{N})$
such that
\begin{itemize}
\item[(1)] if $F,G\in\mathcal{P}_{f}(\mathcal{T})$ and $F\subsetneq G$,
then $\max H(F) < \min H(G)$ and
\item[(2)] whenever $m\in\mathbb{N}$, $G_{1},G_{2},\ldots,G_{m}\in\mathcal{P}_{f}(\mathcal{T})$, 
$G_{1}\subsetneq G_{2}\subsetneq\ldots\subsetneq G_{m}$
and for each $i\in\{1,2,\ldots,m\}$, $f_i \in G_{i}$,
one has\hfill\break
$\prod_{i=1}^{m}(\alpha(G_{i})\ast\prod_{t\in H(G_{i})}f_i(t)\in A$.
\end{itemize}
\end{defn}

\begin{defn}
Let $(S,\ast)$ be a commutative adequate partial semigroup.\\
(a) A set $A\subseteq S$ is a $J_{\delta}$-set if and only if whenever 
$F\in\mathcal{P}_{f}(\mathcal{T})$, $W\in\mathcal{P}_{f}(S)$, there exist
$a\in\sigma(W)$ and $H\in\mathcal{P}_{f}(\mathbb{N})$ such
that for each $f\in F$, $\prod_{t\in H}f(t)\in\sigma(W\ast a)$ and
$a\ast\prod_{t\in H}f(t)\in A$.\\
(b) $J_{\delta}(S)=\{p\in\delta S:$ for all $A\in p\,,\,A$ is a $J_{\delta}$-set$\}$.
\end{defn}

\begin{lem}\label{con}
Let $(S,\ast)$ be a commutative adequate partial semigroup and let $A$ be a $J_{\delta}$-set in $S$. Then for all
$F\in\mathcal{P}_{f}(\mathcal{T})$, and all $r\in\mathbb{N}$ and
$W\in\mathcal{P}_{f}(S)$, there exist $a\in\sigma(W)$ and $H\in\mathcal{P}_{f}(\mathbb{N})$
such that $\min H>r$ and for all $f\in F$, $\prod_{t\in H}f(t)\in\sigma(W\ast a)$
and $a\ast\prod_{t\in H}f(t)\in A$.
\end{lem}

\begin{proof}
Consider $F\in\mathcal{P}_{f}({\mathcal{T}})$, $r\in\mathbb{N}$ and $W\in\mathcal{P}_{f}(S)$. 
For $f\in F$, define $g_{f}\in\mathcal{T}$ by $g_{f}(t)=f(r+t)$,  for all $t\in\mathbb{N}$.
 Since $A$ is a $J_{\delta}$-set, then there exist $a\in\sigma(W)$
and $K\in\mathcal{P}_{f}(\mathbb{N})$ such that $\prod_{t\in K}g_{f}(t)\in\sigma(W\ast a)$ and 
$a\ast\prod_{t\in K}g_{f}(t)\in A$. Using the first condition we get $\prod_{t\in K}f(r+t)\in\sigma(W\ast a)$.
By translating the variable we obtain $\prod_{t\in r+K}f(t)\in\sigma(W\ast a)$.
 Similarly the second condition yields
 $a\ast\prod_{t\in K}f(r+t)\in A$ which turns out to be $a\ast\prod_{t\in H}f(t)\in A $ where $H=r+K$.
\end{proof}

\begin{thm}
Let $(S,\ast)$ be a commutative adequate partial semigroup, let $A\subseteq S$. If there is an idempotent
$p\in\overline{A}\cap J_{\delta}(S)$, then $A$ is a $C$-set.
\end{thm}

\begin{proof}
Let $p$ be an idempotent in $\delta S$ and let $p\in\overline{A}\cap J_{\delta}(S)$. Let 
$A^{\star}=\{a\in A:a^{-1}A\in p\}$, then $A^{\star}\in p$ and by Lemma \ref{1.15}, for every
$a\in A^{\star}$, $a^{-1}A^{\star}\in p$.

We define $\alpha(F) \in S$ and $H(F) \in \mathcal{P}_f(\mathbb{N})$ for $F \in \mathcal{P}_f(\mathcal{T})$ 
by using induction on $|F|$ satisfying the following inductive hypotheses:
\begin{itemize} 
\item[(1)] if $\emptyset\neq G\subsetneq F$, then $\max H(G)<\min H(F)$
and
\item[(2)] if $m\in\mathbb{N}$, $G_{1},G_{2},\ldots,G_{m}\in\mathcal{P}_{f}(\mathcal{T})$,
$G_{1}\subsetneq G_{2}\subsetneq\ldots\subsetneq G_{m}=F$
and for each $i\in\{1,2,\ldots,m\}$, $f_{i}\in G_{i}$, then\hfill\break
$ \prod_{i=1}^{m}(\alpha(G_{i})\ast\prod_{t\in H(G_{i})}f_i(t))\in A^{\star}$.
\end{itemize}

First assume that $F=\{f\}$. Since $A^\star\in p$, $A^\star$ is a $J_{\delta}$-set. 
Pick $W\in \mathcal{P}_{f}(S)$. ($W$ does not enter in to the argument at this stage.)
Since $A^{\star}$ is a $J_{\delta}$-set, pick some
$a\in\sigma(W)$ and $L\in\mathcal{P}_{f}(\mathbb{N})$
such that $\prod_{t\in L}f(t)\in\sigma(W\ast a)$ and $a\ast \prod_{t \in L}f(t)\in A^{\star}$.
Now let $\alpha(F)=a$ and $H(F)=L$. Then hypothesis (1) holds vacuously. To verify hypothesis (2), 
let $n \in \mathbb{N}$, and assume that $\emptyset \neq G_1 \subsetneq  G_2 \subsetneq \ldots \subsetneq G_n=F$, and 
$\langle f_i \rangle_{i=1}^n \in \bigtimes_{i=1}^n G_i$. Then $n=1$ and $f_n=f$ so 
$\prod_{i=1}^{n}(\alpha(G_{i})\ast\prod_{t\in H(G_{i})}f_i(t))= a \ast \prod_{t \in L} f(t) \in A^\star$.

 Now assume that $|F|>1$
and $\alpha(K)$ and $H(K)$ have been chosen for all $K$ with $\emptyset\neq K\subsetneq F$.
Let $r=\max\bigcup\{H(K):\emptyset\neq K\subsetneq F\}$. Let

\[
\begin{array}{rl}
M = & \{\prod_{i=1}^{n}(\alpha(G_{i})\ast\prod_{t\in H(G_{i})}f_{i}(t)) \mid  
n\in\mathbb{N}, G_{1}, G_{2}, \ldots, G_{n} \in\mathcal{P}_{f}(\mathcal{T}),\\
 &\emptyset\neq G_{1}\subsetneq G_{2}\subsetneq\ldots\subsetneq G_{n}\subsetneq F\mbox{ and }
\langle f_i\rangle_{i=1}^n\in\bigtimes_{i=1}^n G_{i}\}.
\end{array}
\]

 Then by hypothesis (2), $M$ is finite and $M\subseteq A^{\star}$.  Let $B=A^{\star}\cap\big(\bigcap_{x\in M}x^{-1}A^{\star}\big)$.
Then $B\in p$ and thus $B$ is a $J_{\delta}$-set.  Pick by Lemma \ref{con}, $a \in \sigma(M)$
and $L \in\mathcal{P}_{f}(\mathbb{N})$ such that $\min L>r$,
and for all $f\in F$, $\prod_{t \in L}f(t) \in \sigma (M \ast a)$ and $a \ast \prod_{t\in L}f(t)\in B$. 
Let $\alpha (F)= a$ and let $H(F)= L$.

Since $\min L > r$,  hypothesis (1) holds. To verify hypothesis (2), let $n \in \mathbb{N}$, let 
$\emptyset \neq G_1 \subsetneq G_2 \subsetneq \ldots \subsetneq G_n=F$, and let $\langle f_i \rangle_{i=1}^n 
\in \bigtimes_{i=1}^n G_i$. If $n=1$, then  $\prod_{i=1}^{n}(\alpha(G_{i})\ast\prod_{t\in H(G_{i})}f_{i}(t))= 
a \ast \prod_{t \in L}f(t) \in B \subseteq  A^\star$.  
Now assume that $n > 1$ and let $y= \prod_{i=1}^{n-1}(\alpha(G_{i})\ast\prod_{t\in H(G_{i})}f_{i}(t))$. 
Then $y \in M$. Therefore,  $a \in \sigma (M) \subseteq \phi(y)$ and  $\prod_{t \in L} f_n(t) \in \sigma(M \ast a) $ 
and $a \ast  \prod_{t \in L} f_n(t) \in B \subseteq y^{-1}A^\star$. Therefore, 
$\prod_{i=1}^n(\alpha(G_i)\prod_{t \in H(G_i)} f_i(t))= y \ast a \ast \prod_{t \in L} f_n(t) \in A^\star$. 
\end{proof}

\vspace{.5cm}

\noindent {\bf Acknowledgement:} The author would like to thank the referee for careful reading of the draft and enormous valuable suggestions to the improvements of the article. She also thanks her advisor Prof. Dibyendu De for his guidance and suggestions. Finally, the author acknowledges support received from the UGC-NET research grant.


\begin{thebibliography}{9}
\vspace{.3cm}
\bibitem{BH90}V. Bergelson and N. Hindman, 
\emph{Nonmetrizable topological dynamics and Ramsey Theory}, Trans.\ Amer.\ Math.\ Soc.\ {\bf 320} (1990), 293-320.

\vspace{.3cm}
\bibitem{DHS08} D.  De,  N.  Hindman, and  D.  Strauss, \emph{A new and stronger central sets theorem},
Fund.\ Math.\ {\bf 199} (2008), no. 2, 155-175.

\vspace{.3cm}
\bibitem{F81} H. Furstenberg, \emph{Recurrence in ergodic theory and
combinatorial number theory}, Princeton University Press, Princeton,
1981.

\vspace{.3cm}
\bibitem{HM01} N. Hindman and  R. McCutcheon, \emph{VIP systems in partial semigroups}, 
Discrete Math. {\bf 240} (2001), 45-70.

\vspace{.3cm}
\bibitem{HS98} N. Hindman and  D. Strauss, \emph{Algebra in the Stone-\v{C}ech
compactification: theory and applications}, W. de Gruyter and Co., Berlin,
1998.

\vspace{.3cm}
\bibitem{HS09} N. Hindman and  D. Strauss, \emph{Sets satisfying the Central sets theorem}, 
Semigroup Forum {\bf 79} (2009), 480-506.

\vspace{.3cm}
\bibitem{M05} J. McLeod, \emph{Central sets in commutative Aadequate partial semigroups}, Topology Proc.\ 
{\bf 29} (2005), no. 2, 567-576.

\vspace{.3cm}
\bibitem{KP17}
 K. Pleasant,\emph{Some new results in Ramsey 
Theory}, Ph.D. Thesis, Howard University, 2017

\end{thebibliography}
\end{document}